\documentclass[a4paper,times]{amsart}

\usepackage{amsmath,amssymb,mathrsfs}

\newtheorem{theorem}{\rm\bf Theorem}[section]

\newtheorem{lemma}[theorem]{\rm\bf Lemma}
\newtheorem{corollary}[theorem]{\rm\bf Corollary}

\newtheorem*{theorem*}{Theorem}
\newtheorem*{theorem 1}{\rm\bf Proposition 1}
\newtheorem*{theorem 2}{\rm\bf Proposition 2}

\theoremstyle{definition}
\newtheorem{definition}[theorem]{\rm\bf Definition}

\theoremstyle{remark}
\newtheorem{remark}[theorem]{\rm\bf Remark}

\def\XXint#1#2#3{{\setbox0=\hbox{$#1{#2#3}{\int}$}
     \vcenter{\hbox{$#2#3$}}\kern-.5\wd0}}

\def\half#1#2{\begin{matrix}\frac{#1}{#2}\end{matrix}}

\def\M{\mathscr{M}}

\def\R#1{\mathbb{R}^{#1}}

\def\scal#1#2{\langle #1; #2 \rangle}

\DeclareMathOperator{\supp}{supp}

\begin{document}

\title[Sharp pointwise gradient estimates for Riesz potentials]{Sharp pointwise gradient estimates for Riesz potentials with a bounded density}

\author{Vladimir G. Tkachev}
\address{Department of Mathematics\\, Link\"oping University\\ 58183,\\ Sweden}

\email{vladimir.tkatjev@liu.se}

\dedicatory{In memory of Sasha Vasil'ev\\Friend, Colleague, Mathematician.}

\begin{abstract}
We establish sharp inequalities for the Riesz potential and its gradient in $\R{n}$ and indicate their usefulness for potential analysis, moment theory and other applications.
\end{abstract}



\maketitle

\section{Introduction}
Given a measurable function  $f(x)$ on $\R{n}$, its Riesz potential of order $0\le \alpha<n$ is defined\footnote{We use a slightly different from the standard normalization} by
$$
(\mathcal{I}_\alpha \rho)(y) =\int_{\R{n}}\frac{\rho(x)}{|y-x|^{n-\alpha}}\; d_\omega x,
$$
where $d_\omega x$ denotes the $n$-dimensional Lebesgue measure on $\R{n}$ normalized by
$$
d_\omega x=\frac{1}{\omega_n}\, dx,
$$
$\omega_n:=\pi^{n/2}/\Gamma(\half{n}{2}+1)$ being the  $n$-dimensional volume of the unit ball in $\R{n}$.
Let $\rho(x)$, $0\leq \rho\leq 1$, be a measurable function with compact support in $\R{n}$ and let
\begin{equation*}\label{intro1}
E_\rho(y)=\exp\left[-2\int_{\R{n}}
\frac{\rho(x)}{|y-x|^{n}}d_\omega x\right]=e^{-2(\mathcal{I}_\alpha \rho)(y)}, \qquad y\in \R{n}\setminus\supp \rho,
\end{equation*}
be the exponential transform of $\rho$ \cite{GP2003}, \cite{GuPubook}. Then $E_\rho(y)$ is the exponential of a Riesz potential of first nonintegrable index ($\alpha=0$).  If $\rho=\chi_D$ is a characteristic function of a bounded domain $D\subset\R{n}$,  Putinar and Gustafsson \cite{GP2003} established that $E_{\chi_{D}}(x)$ is superharmonic in the complement of $D$ and it tends to zero at smooth points of the boundary $\partial D$. It has also been conjectured in \cite{GP2003} that in fact a  stronger result holds: for any density $\rho(x)$, the function
\begin{equation}\label{e-tilda1}
\left\{%
\begin{array}{ll}
    \ln (1-E_\rho(y)), & \hbox{if \quad $n=2$,} \\
    \\
    \frac{1}{n-2}(1-E_\rho(y))^{(n-2)/n}, & \hbox{if \quad $n\geq 3$,} \\
\end{array}%
\right.
\end{equation}
is \textit{subharmonic everywhere outside} $\supp \rho$. The conjecture has been settled in the affirmative by the author in \cite{Tk2005}. A key ingredient in the proof of the subharmonicity was the following sharp inequality.

\begin{theorem}[Corollary~2.2. in \cite{Tk2005}]\label{th:11}
For any  density function $0\le \rho(x)\le 1$,
$0\not\in\supp\rho$, the inequality holds
\begin{equation}\label{var}
\left(\int_{\R{n}}\frac{x_1\rho(x)}{|x|^n}\; d_\omega x\right)^2\leq
\M_n\left(\int_{\R{n}}\frac{\rho(x)}{|x|^n}\; d_\omega x\right)
{\displaystyle\int_{\R{n}}\frac{\rho(x)}{|x|^{n-2}} \;d_\omega x},
\end{equation}
where  $\M_n(t)$ is the unique solution of the  initial problem
\begin{equation}\label{equ:diff}
\M_n'(t)=1-\M_n^{2/n}(t), \qquad \M(0)=0.
\end{equation}
The inequality \eqref{var} is sharp and the equality holds when  $\rho(x)$ is the
characteristic function of a ball $B$ with a center on the $x_1$-axes and $0\not\in \overline{B}$.
\end{theorem}

Inequality \eqref{var} is remarkable in many respects. First notice that it implies a sharp gradient estimate for the Newtonian potential
\begin{equation}\label{gradestimate}
\frac{1}{(n-2)^2}|\nabla \mathcal{I}_2 \rho(x)|^2\le \M(\mathcal{I}_0\rho(x))\cdot \mathcal{I}_2\rho(x),
\end{equation}
or
\begin{equation}\label{equ:potential}
\frac{1}{n-2}|\nabla U_\rho(y)|^2\leq U_\rho(y)\cdot \M_n\bigl(-\half{1}{2}\ln E_\rho(y)\bigr),\qquad\forall y
\not\in \supp \rho,
\end{equation}
for the Newtonian potential
$$
U_\rho(y)=\frac{1}{n-2}\int_{\R{n}} \frac{\rho(x)}{|y-x|^{n-2}} d_\omega x
$$
with a bounded density $\rho$. Since $\M_n(t)<1$, \eqref{equ:potential} yields a  `truncated version'
\begin{equation*}\label{equ:GP}
|\nabla U_\rho(x)|<\sqrt{(n-2) U_\rho(x)}, \qquad x \not\in \supp \rho.
\end{equation*}
It is well-known, see for example Proposition~3.1.7 in \cite{AdamsHed}, that for $f\in L^p(\R{n})$, $1\le p<\infty$, there exists a constant $A$ such that
\begin{equation}\label{gradhedberg}
|\nabla I_\alpha f(x)|^\alpha\le A\cdot Mf(x)\cdot (I_\alpha(f)(x))^{\alpha-1},
\end{equation}
where
$$
Mf(x)=\sup_{r>0}\frac{1}{r^n\omega_n}\int_{B(x,r)}|f(y)|\,dy
$$
is the Hardy-Littlewood maximal function of $f$. As Adams remarks in \cite{Adams}, while the maximal function is an important tool for estimates involving $L^p$ measures $f$, it is not a sharp tool for analyzing their pointwise behaviour. Some generalizations involving the Hardy-Littlewood maximal function for a complex order $\alpha\in \mathbb{C}$ can be found in \cite{MazShap99}.
This estimate also appears  for the Poisson equation and quasilinear equations, see \cite{Mingione11}, \cite{Mingione14}, see also \cite{GargSpector}. In fact, a straightforward application of Cauchy's inequality yields just
\begin{equation}\label{Cauchy}
\left(\int_{\R{n}}\frac{x_1\rho(x)}{|x|^n}\; d_\omega x\right)^2\leq
\int_{\R{n}}\frac{\rho(x)}{|x|^n}\; d_\omega x
\int_{\R{n}}\frac{\rho(x)}{|x|^{n-2}} \;d_\omega x
\end{equation}
which is optimal in the class of \textit{arbitrary} (not uniformly bounded) nonnegative measurable densities $\rho$  and the equality attains asymptotically for a suitable approximation of a $\delta$-function distribution with a single-point support on the $x_1$-axes.

In this respect, \eqref{var} is a considerable refinement of \eqref{Cauchy} for \textit{uniformly bounded} densities $0\le \rho\in L^{\infty}(\R{n})$. Since $\M_n<1$ one readily obtains from \eqref{var} that for any  $0\le \rho\in L^{\infty}(\R{n})$ the (truncated)  Cauchy inequality holds:
\begin{equation}\label{var1}
\left(\int_{\R{n}}\frac{x_1\rho(x)}{|x|^n}\; d_\omega x\right)^2<
\|\rho\|_{L^\infty(\R{n})}\,\int_{\R{n}}\frac{\rho(x)}{|x|^{n-2}} \;d_\omega x,
\end{equation}

Observe that the sharp inequality \eqref{var} has  no longer symmetry of  Cauchy's inequality, see \eqref{Cauchy}. This symmetry breaking  can appropriately be explained in the moment problem context. Namely,  \eqref{var} can be thought as a natural extension of the Markov  inequalities in the  $L$-problem  \cite{AkhKrein}, \cite{KreinNud} for the critical negative powers. Recall that given $L>0$, the $L$-problem of moments  concerns the existence of a density function $0\le \rho\le L$ with a given sequence of the power moments
\begin{equation}\label{sseq}
s_k(\rho):= s_k=\int_{I} \rho(x)x^k dx, \quad k=0,1,2,\ldots,
\end{equation}
where $I\subset \R{}$ is an arbitrary fixed (finite or infinite) interval.
By a celebrated result of A.A.~Markov, the solvability of the $L$-problem is equivalent to the solvability of the corresponding classical moment problem for
$$
a_k=\int_{I} x^kd\mu(x),
$$
where $d\mu(x)$ is a positive measure, and the correspondence is given by the (one-dimensional) {exponential transform}
\begin{equation}\label{exp:def}
1-\exp(-\frac{1}{L}\sum_{k=0}^{\infty}\frac{s_k}{z^{k+1}})=\sum_{k=0}^{\infty}\frac{a_k}{z^{k+1}},
\end{equation}
see \cite[p.~72]{AkhKrein} or \cite[p.~243]{KreinNud} for more details. \cite{GuHeMiPu}. Setting $L=1$ and $I=[0,\infty)$ in \eqref{sseq}, the solvability of the corersponding $L$-moment problem is equivalent to the solvability of the Stieltjes problem for the sequence $\{a_k\}_{k\ge0}$ defined by \eqref{exp:def} which is, in its turn, is  equivalent to the nonnegativity of the Hankel determinants sequence $\Delta_m:=\det (a_{i+j})_{i,j=0}^{m}\geq 0$ and $\Delta_m':=\det (a_{i+j+1})_{i,j=0}^{m}\geq 0$, $m\ge0$. For example, $\Delta'_1\ge0$ readily yields
\begin{equation}\label{Mar1}
s_0^2\le 12(s_0s_2-s_1^2).
\end{equation}
Furthermore, the inversion of the variable $x\to 1/x$ in \eqref{sseq} implies   a correspondence between the negative power moments $s_m$ for $m=-2,-3,\ldots$ and the classical ones:
\begin{equation}\label{sseq1}
s_k(\rho)=s_{-k-2}(\widetilde\rho),\qquad \widetilde\rho(t)=\rho(x^{-1}).
\end{equation}
This implies the classical Markov inequalities for all power moments $s_m$ except for the critical exponent $m=-1$.

In this respect, in the one-dimensional case \eqref{var} provides a novel inequality for power $L$-moments involving the critical exponent $s_{-1}$. Indeed, in the notation of Theorem~\ref{th:11}  one has $\M_1(t)=\tanh t$ and $d_\omega x=\half12 dx$, hence \eqref{var} yields
\begin{equation}\label{tanh}
\frac12\left(\int_{0}^\infty \rho(x) dx\right)^2\le \tanh \left(\frac12\int_{0}^\infty \frac{\rho(x)}{x} dx\right)\cdot \int_{0}^\infty x\rho(x) dx
\end{equation}
for any density function $0\le \rho\le1$. In the moment  notation this yields a sharp inequality
\begin{equation}\label{Lmoment1}
s_0^2\le s_1\tanh s_{-1}.
\end{equation}
Remark that in contrast to the algebraic character of the classical Markov inequalities \cite{CurtoF}, \eqref{Lmoment1} has a different, transcentental nature. The analogous two-dimensional $L$-problem is much less explored, recent works point out some direct applications of this problem to tomography, geophysics, the problem in particular has to do with the distribution of pairs of random variables or the logarithmic potential of a planar domain, see \cite{GuPuSaSt}, \cite{GuPubook}, \cite{Lass}, \cite{GusVas}.

%

\section{Main results}

In this paper we  extend  \eqref{var} on the Riesz potentials of a general index. Then we have for its gradient
\begin{equation}\label{Hgrad}
\frac{1}{n-\alpha}(\nabla \mathcal{I}_\alpha \rho)(y) =\int_{\R{n}}\frac{(y-x)\rho(x)}{|y-x|^{n+2-\alpha}}\; d_\omega x.
\end{equation}
When $y$ is fixed it is natural to assume that $y=0$, hence after a suitable orthogonal transformation of $\R{n}$ the above integrals amount respectively to
\begin{align*}
\mathcal{F}_\alpha\rho&:=\int_{\R{n}}\frac{\rho(x)}{|x|^{n-\alpha}}\; d_\omega x ,\\
\mathcal{H}_\alpha\rho&:=\int_{\R{n}}\frac{x_1\rho(x)}{|x|^{n+2-\alpha}}\; d_\omega x,
\end{align*}
with a new density function $\rho$. We are interested in the gradient estimates, i.e. those involving both $\mathcal{F}_\alpha\rho$ and $\mathcal{H}_\alpha\rho$. In this paper, we consider the following {variational  problem}.

\begin{definition}Given  $u,v>0$, define
\begin{equation}\label{test0}
\mathscr{N}_\alpha(u,v):=\sup_{\rho}\left\{|\mathcal{H}_\alpha\rho|^2: \,\,\mathcal{F}_\alpha\rho=u,
\,\,\,\,\mathcal{F}_{\alpha-2}\rho=v\right\}
\end{equation}
where the supremum is taken over all measurable functions $0\le \rho\le 1$ with support outside of the origin. We refer to such a $\rho(x)$ as an \textit{admissible} density function. A pair $(u,v)\subset\R{2}_{\ge0}$ is said to be \textit{admissible} for the variational problem \eqref{test0} if there exists an admissible density function $\rho$ such that $\mathcal{F}_\alpha\rho=u$ and $\mathcal{F}_{\alpha-2}\rho =v$.
\end{definition}

It is easy to see that $\mathscr{N}_\alpha(u,v)$ is well-defined and finite for any $\alpha$ and any admissible pair $(u,v)$. Indeed, it follows from Cauchy's inequality that
\begin{equation}\label{Nuv}
\mathscr{N}_\alpha(u,v)\le uv, \qquad \forall \alpha\in \R{}.
\end{equation}
We point out, however, that the estimate \eqref {Nuv} provide a correct approximation only when $u$ and $v$ are infinitesimally small.

By virtue of \eqref{Hgrad},  $\mathscr{N}_\alpha(u,v)$ implies the following pointwise gradient estimate on the Riesz potential $\mathcal{I}_\alpha \rho$ by means of $\mathcal{I}_\alpha \rho$ itself and the contiguous potential $\mathcal{I}_{\alpha-2} \rho$.

\begin{corollary}
In the above notation, the following pointwise estimate holds:
\begin{equation}\label{gradientest}
\frac{1}{n-\alpha}|\nabla \mathcal{I}_\alpha\rho|\le \sqrt{\mathscr{N}_\alpha\bigl(\mathcal{I}_\alpha\rho,\mathcal{I}_{\alpha-2}\rho \bigr)},
\end{equation}
and the inequality is sharp.
\end{corollary}

Our main result provides an explicit form of the goal function $\mathscr{N}_\alpha(u,v)$.

\begin{theorem}\label{th:N}
Let $n\ge 1$ and $0<\alpha\le 2$. Then the set of admissible pairs coincides with the nonnegative quadrant $\R{2}_{\ge0}$ and for any  $u,v>0$
\begin{equation}\label{Nexplicit}
\mathscr{N}_\alpha(u,v)=u^{2(\alpha-1)/\alpha}\frac{h^2_\alpha(t)} {f_\alpha^{2(\alpha-1)/\alpha}(t)},
\end{equation}
where $t=t(u,v)$ is uniquely determined by the relation
\begin{equation}\label{fft}
f_{\alpha}^{2-\alpha}(t)f^{\alpha}_{\alpha-2}(t)=u^{2-\alpha}v^{\alpha},
\end{equation}
where
\begin{align*}
f_\alpha(t)&=t^{2-n}(t^2-1)^{n/2}F(\half{2-\alpha}{2},\half{2+\alpha}{2};\half{n+2}{2},1-t^2)
\\
h_\alpha(t)&=t^{1-n}(t^2-1)^{n/2}F(\half{2-\alpha}{2}, \half\alpha2;\half{n+2}2,1-t^2),
\end{align*}
and $F([a,b],[c],t)$ is the Gauss hypergeometric function.
\end{theorem}

We are in particular interested in  the \textit{shape structute} of the goal function $\mathscr{N}_\alpha(u,v)$, i.e. how it depends on the variables $u$ and $v$. Combining \eqref{fft} and  \eqref{Nexplicit} yields  the following alternative representation.

\begin{corollary}\label{cor:th21}
Under assumptions of Theorem~\ref{th:N},
\begin{equation}\label{NNalpha}
\mathscr{N}_\alpha(u,v)=uv\cdot \psi(u^{2-\alpha}v^{\alpha}),
\end{equation}
where $\psi(s)$ well-defined by  the parametric representation
$$
\psi(s)=\frac{h_\alpha^2(t)}{f_\alpha(t)f_{\alpha-2}(t)}, \qquad s=f_{\alpha}^{2-\alpha}(t)f^{\alpha}_{\alpha-2}(t).
$$
\end{corollary}

Some further remarks are in order. The condition $\alpha\le 2$ in Theorem~\ref{th:N} is of a technical character. In the complementary case $2<\alpha<n$, the set of admissible pairs is a proper subset of the quadrant $\R{2}_{\ge0}$. The corresponding analysis requires some more care, and will be done elsewhere.

The borderline case $\alpha=2$ corresponds to inequality  \eqref{var} established in \cite{Tk2005} and in the present notation the goal function here becomes
\begin{equation}\label{NM2}
\mathscr{N}_2(u,v)=u\cdot \mathscr{M}_n(v).
\end{equation}
We derive also it as an application of Theorem~\ref{th:N} in Corollary~\ref{cor:alpha2} below. Remarkably, that both Cauchy's inequality estimate \eqref{Nuv} and its sharp version \eqref{NM2}  separate into  one-variable factors. This separable form, however, no longer holds for a general $\alpha$, when the shape of $\mathscr{N}_\alpha$ has a more complicated structure, see \eqref{NNalpha}.

Another interesting case is $\alpha=1$. Under this condition,  the Riesz potentials in the right hand side of \eqref{gradientest} have in fact the same exponent because for $\alpha=1$  the contiguous potential  amounts to $\mathcal{I}_{-1} \rho=\mathcal{I}_{1} \widetilde\rho$, where   $\widetilde\rho$ is  the inversion of $\rho$. A further remarkable feature of this case is that $\mathscr{N}_1(u,v)$ becomes a symmetric function of $u$ and $v$. Indeed, a straightforward analysis  of \eqref{NNalpha} implies that for  $\alpha=1$ the  goal function $\mathscr{N}_1(u,v)$ depends only on the product $uv$. More precisely, we have the following

\begin{theorem}
  \label{th:main}
For any measurable function $0\le \rho(x)\le 1$,
$0\not\in\supp\rho$, the sharp inequality holds
\begin{equation}\label{var1}
\left|\int_{\R{n}}\frac{x_1\rho(x)}{|x|^{n+1}}\; d_\omega x\,\right|\leq
\Phi_n\left(\sqrt{
\int_{\R{n}}\frac{\rho(x)}{|x|^{n-1}}\; d_\omega x
\cdot \int_{\R{n}}\frac{\rho(x)}{|x|^{n+1}} \;d_\omega x}\right),
\end{equation}
where  $\Phi_n(t)$ is the unique solution of the  initial problem
\begin{equation}\label{Phi1}
\Phi_n''=\frac{\Phi_n'(\Phi_n'^2-1)}{(n-1)\Phi_n\Phi_n'+s}, \qquad \Phi_n(0)=0, \,\,\, \Phi_n'(0)=1
\end{equation}
subject to the asymptotic condition
\begin{equation}\label{Phi10}
\lim_{s\to \infty}\frac{\Phi_n(s)}{ \ln s}=\frac{\Gamma(\half{n+2}2)} {\Gamma(\half{n+1}2)\Gamma(\half32)}.
\end{equation}

\end{theorem}

\begin{remark}
Concerning the definition of the shape function $\Phi_n$, we note that the initial problem \eqref{Phi1} itself does not determine a unique solution because the initial condition $\Phi'_n(0)=1$ is  singular. One can prove that if $\phi(x)$ solves \eqref{Phi1} then any solution of \eqref{Phi1} is obtained by the homothetic scaling $\frac{1}{c}\phi(cx)$, $c>0.$ Note also that all thus obtained solutions have the logarithmic growth at infinity, so a further normalization like \eqref{Phi10} is needed.
\end{remark}

The proof of Theorem~\ref{th:N} relies on a refinement  of the technique initiated in \cite{Tk2005} and uses the Bathtub principle. We obtain some preliminary results in section~\ref{sec:prelim} and finish the proof in section~\ref{sec:lemma}. Then we prove Theorem~\ref{th:main} in section~\ref{sec:proofcor}.

Remark also that the obtained gradient estimates are sharp for Euclidean balls with constant density. The latter symmetry phenomenon is natural for the Riesz and Newton potentials  \cite{Fraenkel}, \cite{MazyaSobolev}, and studied recently in connection with Riesz potential integral equations on  exterior domains \cite{Reichel09}, \cite{HuangHL}, \cite{Huang}, \cite{XiaoJ}.

We finally mention that our results can also be thought of as an analogue of the polynomial moment inequalities for the singular Riesz potential $d\mu_\alpha(x)=|x|^{\alpha-n}$ in $\R{n}$. Then the above functionals are recognized as the lower order polynomial moments:
$$
\mathcal{F}_\alpha\rho=\int 1d\mu_\alpha(x),\quad
\mathcal{H}_\alpha\rho=\int x_1d\mu_\alpha(x),\quad
\mathcal{F}_{\alpha-2}\rho=\int (x_1^2+\ldots+x_n^2) d\mu_\alpha(x)
$$
It is natural to speculate what is the natural extension of the Hankel determinant inequalities for $d\mu_\alpha$. We pursue this theme elsewhere.

\section{Auxiliary identities for spherical integrals}\label{sec:prelim}
Let us decompose $\R{n}=\R{}\times \R{n-1}$ such that $x=(x_1,y)$, where $y=(x_2,\ldots,x_n)\in\R{n-1}$. Given $0<\sigma<\tau$, let $B(\tau,\sigma)$ denote the open ball of radius $\sqrt{\tau^2-\sigma^2}$ centered at $(\tau,0)\in \R{n}$, i.e.
$$
B(\tau,\sigma)=\{x=(x_1,y)\in \R{n}: |x|^2-2\tau x_1+\sigma^2<0\},
$$
and let $D(t):=B(t,1)$. We refer to $B(\tau,\sigma)$ as an $x_1$-\textit{ball}.   It is easy to see and will be used later that the inversion $x\to x^*=x/|x|^2$ acts on $x_1$-{balls as follows:
\begin{equation}\label{invers1}
B(\tau,\sigma)^*=B(\half{\tau}{\sigma^2}, \half{1}{\sigma}), \qquad D(t)^*=D(t).
\end{equation}
Let us fix some notation:
\begin{equation}\label{HF}
\begin{split}
F_\alpha(\tau,\sigma)&\equiv \mathcal{F}_\alpha\chi_{B(\tau,\sigma)}= \int_{B(\tau,\sigma)}\frac{1}{|x|^{n-\alpha}}\,d_\omega x =\sigma^{\alpha}f_\alpha(\half{\tau}{\sigma})\\
H_\alpha(\tau,\sigma)&\equiv \mathcal{H}_\alpha\chi_{B(\tau,\sigma)}= \int_{B(\tau,\sigma)}\frac{x_1}{|x|^{n+2-\alpha}}\,d_\omega x=\sigma^{\alpha-1} h_\alpha(\half{\tau}{\sigma})
\end{split}
\end{equation}
where
$$
f_\alpha(t):=F_\alpha(t,1),\qquad h_\alpha(t):=H_\alpha(t,1).
$$
All the introduced functions  depend also on the ambient dimension $n$.

First notice that  $f_\alpha(t)$ and $h_\alpha(\alpha)$ are real analytic functions of  $t>0$, and for any real $\alpha$
\begin{equation}\label{inf}
\lim_{t\to\infty}f_\alpha(t)=\int_{x_1>0}\frac{ d_\omega x}{|x|^{n-\alpha}}=\infty.
\end{equation}
Applying   Stoke's formula  to
\begin{align*}
0&=\int_{B(\tau,\sigma)}\mathrm{div} \left(\frac{|x|^2-2\tau x_1+\sigma^2}{|x|^{n-\alpha+2}}\cdot x\right)d_\omega x
\end{align*}
we get the following identity:
\begin{equation}\label{ident1}
\alpha F_{\alpha}(\tau,\sigma)+(\alpha-2)\sigma^2F_{\alpha-2}(\tau,\sigma)= 2(\alpha-1)\tau H_{\alpha}(\tau,\sigma).
\end{equation}
Further, applying the inversion readily yields  by virtue of \eqref{invers1} that
\begin{equation}\label{ident2}
F_{\alpha}(\tau,\sigma)=F_{-\alpha}(\half{\tau}{\sigma^2}, \half{1}{\sigma}),
\qquad
H_{\alpha}(\tau,\sigma)=H_{2-\alpha}(\half{\tau}{\sigma^2}, \half{1}{\sigma}).
\end{equation}
For the reduced functions this amounts to
\begin{equation}\label{ident1a}
\begin{split}
2(\alpha-1)th_{\alpha}(t)&=\alpha f_{\alpha}(t)+(\alpha-2)f_{\alpha-2}(t),\\
f_\alpha(t)&=f_{-\alpha}(t),\\
h_{\alpha}(t)&=h_{2-\alpha}(t)
\end{split}
\end{equation}

\begin{lemma}
For any $\alpha\in \R{}$, the following identities hold:
\begin{align}
2th_{\alpha}'\,\,\,&=f'_\alpha+f'_{\alpha-2}, \label{coarea1}\\
(n-\alpha)h_{\alpha}&=(t^2-1)f_\alpha'-\alpha tf_\alpha, \label{coarea2}\\
\lim_{t\to 1+0}f_\alpha(t)(1-t^2)^{-n/2}&=1.\label{coarea3}
\end{align}
\end{lemma}

\begin{proof}
Let us consider an axillary function $\lambda(x)=(|x|^2+1)/2x_1$. Then $x_1\nabla \lambda+\lambda e_1=x$, hence
\begin{equation}\label{lamb}
\nabla \lambda=\frac{x-\lambda e_1}{x_1}
\end{equation}
and
\begin{equation}\label{nabla1}
x_1^2|\nabla \lambda|^2=|x|^2-2\lambda x_1+\lambda^2=\lambda^2-1.
\end{equation}
Notice that the $s$-level set $\{x\in \R{n}:\lambda(x)=s\}$ is exactly the boundary sphere $\partial D(s)$, therefore $\lambda$ foliates the punctured ball $D(t)\setminus \{(1,0)\}$ into the family of  spheres $\{\partial D(s): 1< s\le t\}$. Applying the co-area formula one obtains from \eqref{HF} and \eqref{nabla1}
$$
 f_\alpha(t)=\frac{1}{\omega_n}\int_{1}^t ds\int_{\partial D(s)}\frac{dA}{|x|^{n-\alpha}|\nabla \lambda(x)|} =\int_{1}^t \frac{ds}{\sqrt{s^2-1}}\int_{\partial D(s)}\frac{x_1}{|x|^{n-\alpha}}d_\omega A
$$
where $dA$ is the $(n-1)$-dimensional surface measure on $\partial D(s)$ and $d_\omega A=\frac{1}{\omega_n} dA$.
In particular,
\begin{equation}\label{diff1}
f_\alpha'(t)=\frac{1}{\sqrt{t^2-1}}\int_{\partial D(t)}\frac{x_1}{|x|^{n-\alpha}}d_\omega A.
\end{equation}
Similarly one finds
\begin{align*}
h_{\alpha}(t)&=\int_{1}^t \frac{dt}{\sqrt{t^2-1}}\int_{\partial D(s)}\frac{x_1^2}{|x|^{n+2-\alpha}}d_\omega A\\
&=
\int_{1}^t \frac{dt}{2t\sqrt{t^2-1}}\int_{\partial D(s)}\frac{x_1(|x|^2+1)}{|x|^{n+2-\alpha}}d_\omega A.
\end{align*}
Differentiating the obtained identity and applying \eqref{diff1} yields \eqref{coarea1}.

Next, $\lambda|_{\partial D(s)}=s$, therefore the outward normal vector $\nu$ along the boundary $\partial D(s)$ is found from \eqref{lamb} by
$$
\nu=\frac{\nabla\lambda}{|\nabla \lambda|}=\frac{x-se_1}{\sqrt{s^2-1}}
$$
thus, using identities
$$
\mathrm{div}\frac{x}{|x|^{n-\alpha}}=\frac{\alpha}{|x|^{n-\alpha}}, \qquad
\mathrm{div}\frac{e_1}{|x|^{n-\alpha}}=-\frac{(n-\alpha)x_1}{|x|^{n+2-\alpha}},
$$
and applying Stoke's formula  we obtain by virtue of \eqref{diff1}
\begin{align*}
\alpha tf_\alpha(t)+(n-\alpha)h_{p+2}(t)&=
\int_{D(t)}\left(\frac{\alpha t}{|x|^{n-\alpha}}-\frac{(n-\alpha)x_1}{|x|^{n+2-\alpha}}\right)d_\omega x
\\
&=\int_{\partial D(t)}\frac{\scal{tx-e_1}{\nu}}{|x|^{n-\alpha}}d_\omega A
\\
&=\sqrt{t^2-1}\int_{\partial D(t)}\frac{x_1d_\omega A}{|x|^{n-\alpha}}\\
&=(t^2-1)f_{\alpha}'(t)
\end{align*}
which proves \eqref{coarea2}. Finally, it follows from \eqref{diff1} that
\begin{align*}
\lim_{t\to1+0}f_\alpha'(t)(t^2-1)^{(2-n)/2}
&=\lim_{t\to1+0}\frac{1}{(t^2-1)^{(n-1)/2}}\int_{\partial D(t)}\frac{x_1}{|x|^{n-\alpha}}d_\omega A\\
&=\frac{1}{\omega_n}\lim_{t\to1+0}|\partial D(t)|\cdot (t^2-1)^{(1-n)/2}=n.
\end{align*}
Then \eqref{coarea3} follows from $f_\alpha(1)=0$ and the previous identity by virtue of l'Hospital's rule.
\end{proof}

\begin{lemma}\label{lem:est}
If $0<\gamma<1$ then for any $p,q$ there holds that
\begin{equation}\label{fpp}
 f_p^\gamma(t) f_q^{1-\gamma}(t)\ge f_{\gamma p+(1-\gamma)q}(t).\end{equation}
\end{lemma}

\begin{proof}
A straightforward corollary of \eqref{HF} and the  H\"older inequality.
\end{proof}
\medskip

\section{The reduced functions via hypergeometric functions}
Differentiating the first identity in \eqref{ident1a} followed by elimination of $f_{\alpha-2}$, $f'_{\alpha-2}$ and $h_{\alpha}'$ by virtue of \eqref{coarea1} and \eqref{coarea2} readily yields the following identity:
\begin{equation}\label{ODEp}
t(t^2-1)f_{\alpha}''+(t^2-n+1)f'_{\alpha}-t\alpha^2f_{\alpha}=0.
\end{equation}
Setting
$$
f_\alpha(t)=(t^2-1)^{n/2}\phi_\alpha(t^2),
$$
the equation \eqref{ODEp} is transformed to the hypergeometric differential equation
\begin{equation}\label{ODEhyp}
z(1-z)\phi_\alpha''(z)+\biggl(\frac{n}{2}-(1+n)z\biggr)\phi_\alpha'(z) -\frac{(n+\alpha)(n-\alpha)}{4}\phi_\alpha(z)=0.
\end{equation}
with
$$
a=\frac{n-\alpha}{2}, \quad b=\frac{n+\alpha}{2}, \quad c=\frac{n}{2}.
$$
By \eqref{coarea3}, $\phi_{\alpha}(1)=1$. Since $\phi_{\alpha}(z)$ is regular at $z=1$, it follows from a Kummer transformation (see formula 15.5.5 in \cite[p.~563]{AbramS}) that
\begin{equation}\label{hyper}
\phi_{\alpha}(z)=F(\half12(n-\alpha), \half12(n+\alpha);\half12n+1,1-z),\quad z\ge1.
\end{equation}
 Using another Kummer transformation (formula 15.5.11 in \cite[p.~563]{AbramS}) yields an alternative representation
\begin{equation}\label{hyper1}
\phi_{\alpha}(z)=z^{-\frac{n+\alpha}{2}}F(\half{n+\alpha}2, \half{2+\alpha}2;\half{n+2}2,\half{z-1}{z}),\quad z\ge1.
\end{equation}
The latter representation is useful for the asymptotic behavior of $f_\alpha$ at $\infty$. From \eqref{hyper1} we obtain
\begin{equation}\label{hyper2}
f_\alpha(t)=t^{-n-\alpha}(t^2-1)^{n/2}F(\half{n+\alpha}2, \half{2+\alpha}2;\half{n+2}2,\frac{t^2-1}{t^2}).
\end{equation}
In particular, using the Gauss type identity
\begin{equation}\label{hyper3}
\lim_{z\to 1}(1-z)^{a+b-c}F(a,b;c;z)=\frac{\Gamma(c)\Gamma(a+b-c)} {\Gamma(a)\Gamma(b)}, \quad c<a+b
\end{equation}
we have from \eqref{hyper2} the following asymptotic growth:
\begin{equation}\label{hyper4}
\lim_{t\to \infty}\frac{f_\alpha(t)}{t^\alpha}=\frac{\Gamma(\half12n+1)\Gamma(\alpha)} {\Gamma(\half12(n+\alpha))\Gamma(\half12\alpha+1)}.
\end{equation}
Also, applying a linear transformation  (formula 15.3.5 in \cite[p.~559]{AbramS}) to \eqref{hyper2} yields
\begin{equation}\label{hyper1lin}
f_\alpha(t)=t^{2-n}(t^2-1)^{n/2}F(\half{2-\alpha}{2}, \half{2+\alpha}{2};\half{n+2}{2},1-t^2).
\end{equation}

A similar argument also works  for $h_\alpha$:  eliminating $f_{\alpha-2}$ from \eqref{coarea1} by virtue of \eqref{ident1a}${}_1$ yields
\begin{equation}\label{fheq}
f_\alpha(t)=\frac{t^2-1}\alpha h_\alpha'(t)+\frac1{\alpha t} ((\alpha-1)t^2+1-n)h_\alpha(t),
\end{equation}
therefore \eqref{coarea1} amounts to
$$
t^2(t^2-1)h_\alpha''(t)
+t(t^2+1-n)h_\alpha'(t)+
(n-1-(\alpha-1)^2t^2)h_\alpha(t)=0.
$$
The a substitution $h_\alpha(t)=t(t^2-1)^{n/2}\psi_\alpha(t^2)$ transforms the latter equation into a hypergeometric one:
\begin{equation}\label{ODEhyp1}
z(1-z)\psi_\alpha''(z)+\biggl(1+\frac{n}{2}-(n+2)z\biggr)\psi_\alpha'(z) -\frac{(n+\alpha)(n+2-\alpha)}{4}\psi_\alpha(z)=0
\end{equation}
with
$$
a'=\frac{n+2-\alpha}{2}, \quad b'=\frac{n+\alpha}{2}, \quad c'=\frac{n+2}{2}.
$$
For the same reasons as above, we obtain
\begin{equation}\label{hyperpsi}
\psi_{\alpha}(z)=F(\half12(n+2-\alpha), \half12(n+\alpha);\half12n+1,1-z),\quad z\ge1,
\end{equation}
and
\begin{equation}\label{hyperpsi2}
h_\alpha(t)=t^{1-n}(t^2-1)^{n/2}F(\half{2-\alpha}{2}, \half\alpha2;\half{n+2}2,1-t^2).
\end{equation}

Let us consider some particular cases when $f_\alpha$ can be explicitly specified.
When $n=1$ and $\alpha$ is arbitrary, one easily finds from \eqref{HF} that
\begin{equation}\label{explicit1}
\begin{split}
f_\alpha(t)&=\left\{
         \begin{array}{ll}
           \frac{1}{\alpha}\sinh \alpha\xi, & \hbox{if $\alpha\ne 0$;} \\
           \xi, & \hbox{if $\alpha=0$,}
         \end{array}
       \right.,
\end{split}
\end{equation}
and $h_{\alpha}(t)=f_{\alpha-1}(t)$, where
\begin{equation}\label{xieq}
t=\cosh \xi, \quad 0\le \xi<\infty.
\end{equation}
This yields
\begin{align*}
f_\alpha(t)&=\frac{\biggl(t+\sqrt{t^2-1}\biggr)^\alpha - \biggl(t-\sqrt{t^2-1}\biggr)^\alpha}{2\alpha}
\end{align*}

When $\alpha=2$,   the spherical mean property for harmonic functions  $|x|^{2-n}$ and $x_1|x|^{-n}$ was used in \cite[Sec.~2.2]{Tk2005} to obtain explicit expressions
\begin{equation}\label{spheric1}
f_{2}(t)=(t^2-1)^{n/2}t^{2-n}, \qquad
h_{2}(t)=\half{1}{t}f_2(t)
\end{equation}
and
\begin{equation}\label{spheric2}
f_{0}(t)=\int_{1}^t (s^2-1)^{(n-2)/2}s^{1-n}\,ds.
\end{equation}

Another interesting particular case is $\alpha=1$, we have by  \eqref{ident1a}${}_1$ that $f_{-1}=f_1$, and the  reduced functions $f_1$ and $h_1$ can be determined explicitly at least when $n$ is an odd integer.

\section{Proof of Theorem~\ref{th:N}}\label{sec:lemma}



\textbf{Step 1. } First we assume that $\supp\rho(x)\subset \R{n}_{x_1>0}$ and let $\mathscr{N}^+_\alpha(u,v)$ denote the corresponding supremum in \eqref{test0}.
We claim that for any $u,v>0$  there exist $0<\sigma<\tau$ such that
\begin{equation}\label{systemFH}
{F}_\alpha(\tau,\sigma)=u, \quad
{F}_{\alpha-2}(\tau,\sigma)=v,
\end{equation}
in other words, the pair $(u,v)$ is admissible by a characteristic function of an $x_1$-centered ball.
Indeed, rewrite  \eqref{systemFH} by virtue of the reduced functions as the system
\begin{equation}\label{systemFH1}
\left\{
\begin{array}{rcl}
f_\alpha(t)&=&u\sigma^{-\alpha}, \\
f_{\alpha-2}(t)&=&v\sigma^{2-\alpha}.
\end{array}
\right.
\end{equation}
where $t=\half{\tau}{\sigma}$. Consider an auxiliary function
\begin{equation}\label{funcg}
g(t)=f_{\alpha}^{2-\alpha}(t)f^{\alpha}_{\alpha-2}(t).
\end{equation}
It follows from $0<\alpha<2$ that $g(t)$ is an increasing function of $t$ and by virtue \eqref{coarea3} we have $\lim_{t\to 1+0}g(t)=0$. Furthermore, setting $\gamma=\frac{2-\alpha}{2}$, $p=\alpha$ and $q=\alpha-2$ in \eqref{fpp}  yields under the made assumptions that
$$
g(t)=f_{\alpha}^{2-\alpha}(t)f^{\alpha}_{\alpha-2}(t)\ge f_{0}^2(t),
$$
hence  \eqref{inf} implies that $\lim_{t\to\infty} g(t)=\infty$. Thus, $g$ is a  bijection of $[1,\infty)$ onto $[0,\infty)$.

Now, let $t=t_0$ be the unique solution of $g(t)=u^{2-\alpha}v^{\alpha}$ and let
\begin{equation}\label{sigmadef}
\sigma_0:=f^{-\frac{1}{\alpha}}_\alpha(t_0)u^{\frac{1}{\alpha}}.
\end{equation}
Then it follows from \eqref{systemFH1} and \eqref{ident2} that $\sigma_0$ and $\tau_0=\sigma_0 t_0$ is a (unique) solution of \eqref{systemFH}. This proves our claim, and also implies that the set of admissible pairs $(u,v)$ coincides with the nonnegative quadrant $\R{2}_{\ge0}$.

In the introduced  notation, let $\rho_0(x)=\chi_{B(\tau_0,\sigma_0)}(x).$
Then $\rho_0(x)$ is a test function for \eqref{test0}. Thus, using \eqref{HF} we find
$$
\mathscr{N}^+_\alpha(u,v)\ge (\mathcal{H}_\alpha\rho_0)^2 =\sigma_0^{2(\alpha-1)}h^\alpha_{2}(t_0).
$$
On the other hand, if $\rho(x)$ is an arbitrary test function for \eqref{test0} then by our choice, both $\rho$ and $\rho_0$ are test functions in the auxiliary problem
\begin{equation}\label{test1}
A=\sup\,\{\int \rho(x)\phi(x)d\mu(x): \,\,\int \rho(x)d\mu(x)=\sigma_0^2 u+v, \,\,\,\, 0\le \rho\le1\},
\end{equation}
where $\phi(x)=\frac{x_1}{|x|^{2}+\sigma_0^2}$ and $d\mu(x)=(\frac{\sigma_0^2}{|x|^{n-\alpha}}+\frac{1}{|x|^{n+2-\alpha}}) d_\omega x$. In particular, we have
$$
\mathscr{N}^+_\alpha(u,v)\le A^2.
$$
By the `Bathtub principle' \cite[p.~28]{Lieb}, a  solution of the  variational problem \eqref{test1} is given by the characteristic function of a sublevel set $\{x\in \R{n}: \phi(x)\le \frac{1}{2\tau}\}\equiv B(\tau,\sigma_0)$ where $\tau$ is uniquely determined by the test condition
\begin{equation}\label{Bt}
\int \chi_{B(\tau,\sigma_0)}(x)d\mu(x)=\sigma_0^2 u+v.
\end{equation}
Since the latter integral is an increasing function of $\tau$ and since $\tau_0$ satisfies \eqref{Bt}, we conclude that $\rho_0=\chi_{B(\tau,\sigma_0)}$ is a maximizer in \eqref{test1}. This yields
$$
A=\int \rho_0(x)\phi(x)d\mu(x)=\int\frac{x_1\rho_0}{|x|^{n+2-\alpha}}d_\omega x.
$$
Combining the obtained inequalities and using \eqref{sigmadef} implies
\begin{equation}\label{Nuv0}
\mathscr{N}^+_\alpha(u,v)=\sigma_0^{2(\alpha-1)}h^2_{\alpha}(t_0)
\equiv u^{2(\alpha-1)/\alpha}\frac{h^2_\alpha(t_0)} {f_\alpha^{2(\alpha-1)/\alpha}(t_0)}.
\end{equation}

\textbf{Step 2.} We claim that $\mathscr{N}^+_\alpha(u,v)$ defined implicitly by~\eqref{Nuv0} is an increasing function of each argument separately. It suffices to verify this for an auxiliary function $G(u,v)=(\mathscr{N}^+_\alpha(u,v))^{\alpha/2}$. On eliminating $\sigma_0$ by virtue of \eqref{systemFH1} we obtain
\begin{equation}\label{g1g2}
G(u,v)=u^{\alpha-1}g_1(t_0) =v^{\frac{\alpha(1-\alpha)}{2-\alpha}}g_2(t_0)^{\frac{\alpha}{2-\alpha}},
\end{equation}
where
$$
g_1(t)=\frac{h_\alpha^{\alpha}(t)}{f_{\alpha}^{\alpha-1}(t)}, \quad
g_2(t)=\frac{h_\alpha^{2-\alpha}(t)}{f_{\alpha-2}^{1-\alpha}(t)}.
$$
Using \eqref{ident1a} and \eqref{coarea1} yields for the logarithmic derivatives
\begin{align*}
\frac{g_1'}{g_1}&=\frac{\alpha h_\alpha'f_\alpha-(\alpha-1) f_\alpha'h_\alpha}{f_\alpha h_\alpha}=
\frac{\alpha f_{\alpha-2}'f_\alpha-(\alpha-2) f_\alpha'f_{\alpha-2}}{2tf_\alpha h_\alpha}= \frac{g'}{g}\cdot\frac{f_{\alpha-2}}{2th_\alpha}>0
\\
\frac{g_2'}{g_2}&=\frac{(2-\alpha) h_\alpha'f_{\alpha-2}+(\alpha-1) f_{\alpha-2}'h_\alpha}{f_{\alpha-2} h_\alpha}=
\frac{\alpha f_{\alpha}h'_\alpha-(\alpha-1) f_\alpha'h_{\alpha}}{f_{\alpha-2} h_\alpha}= \frac{g'_1}{g_1}\cdot\frac{f_\alpha}{f_{\alpha-2}}>0
\end{align*}
which implies that $g_1$ and $g_2$ are increasing functions of $t$ for any $0<\alpha\le 2$.
Now suppose that $u_1\ge u$ and $v_1\ge v$. Let $t_1$ be the unique solution of $g(t_1)=u_1^{\alpha-2}v_1^\alpha$. Since $g(t)$ in \eqref{funcg} is increasing we conclude that $t_1\ge t$. First let us consider $1\le \alpha\le 2$. Then using the first equality in \eqref{g1g2} yields
$$
G(u,v)=u^{\alpha-1}g_1(t_0) \le
u_1^{\alpha-1}g_1(t_1)=G(u_1,v_1).
$$
Similarly, if $0< \alpha\le 1$ then using the second equality in \eqref{g1g2} yields
$$
G(u,v)= v^{\frac{\alpha(1-\alpha)}{2-\alpha}}g_2(t_0)^{\frac{\alpha}{2-\alpha}} \le
v_1^{\frac{\alpha(1-\alpha)}{2-\alpha}}g_2(t_1)^{\frac{\alpha}{2-\alpha}}= G(u_1,v_1).
$$
The claim follows.

\textbf{Step 3. }
Now suppose that $\rho$ is an arbitrary density with a compact support not containing the origin. We claim that $\mathscr{N}^+_\alpha(u,v)\le \mathscr{N}_\alpha(u,v)$. Let $U_{\pm}$ and $U_{0}$ denote the classes of densities satisfying respectively the conditions

(a) $\supp\rho\cap \R{n}_{\mp x_1>0}$ has measure zero;

(b) both $\supp\rho\cap \R{n}_{x_1>0}$ and $\supp\rho\cap \R{n}_{x_1<0}$ have nonzero measures.

\noindent
Let $\mathscr{N}^{*}_\alpha(u,v)$ denote the supremum in \eqref{test0} taken over the corresponding class of densities. Then
$$
\mathscr{N}_\alpha(u,v)= \max \{\mathscr{N}^-_\alpha(u,v), \mathscr{N}^+_\alpha(u,v), \mathscr{N}^0_\alpha(u,v)\},
$$
Let $x\to\tilde{x}$ be the reflection  about the hyperplane $x_1=0$. Then $\widetilde{\rho}(x):=\rho(\tilde{x})$ is a bijection between $U_+$ and $U_-$. This implies   $\mathscr{N}^-_\alpha(u,v)=\mathscr{N}^+_\alpha(u,v)$. Let $\rho\in U^0$ and let $\rho^\pm(x)=\chi_{\R{n}_{\pm x_1>0}}\rho(x)$ such that $\rho^\pm(x)\in U^{\pm}$ and $\rho=\rho^++\rho^-$ a.e. in $\R{n}$. We have
\begin{align*}
|\mathcal{H}_\alpha(\rho)|^2&=|\mathcal{H}_\alpha(\rho^+) -\mathcal{H}_\alpha(\widetilde{\rho^-})|^2
\le \max\{\mathcal{H}_\alpha(\rho^+)^2 ,\,\mathcal{H}_\alpha(\widetilde{\rho^-})^2\}\\
&\le
\mathcal{H}_\alpha(\rho_1)^2 \le \mathscr{N}^+_\alpha(u_1,v_1),
\end{align*}
where $\rho_1$ is one of $\rho^+$ and $\widetilde{\rho^-}$, and
$\mathcal{F}_\alpha\rho_1=u_1\le u$, $\mathcal{F}_{\alpha-2}\rho_1 =v_1\le v.$ By the  monotonicity of $\mathscr{N}^+_\alpha$,
$$
|\mathcal{H}_\alpha(\rho)|^2\le
\mathscr{N}^+_\alpha(u,v)
$$
implying $\mathscr{N}^0_\alpha(u,v)\le \mathscr{N}^+_\alpha(u,v)$. This proves our claim and, thus, finishes the proof of the theorem.

\section{Some applications}


First we demonstrate how Theorem~\ref{th:N} implies Theorem~\ref{th:11}. To this end consider $\alpha=2$.

\begin{corollary}
\label{cor:alpha2}
$$\mathscr{N}_2(u,v)=u\M_n(v)
$$where $\M_n(t)$ is the unique solution of the  initial problem \eqref{equ:diff}.
\end{corollary}

\begin{proof}
By \eqref{hyper1lin},
\begin{equation}\label{0001}
f_2(t)=t^{2-n}(t^2-1)^{n/2}F(0, 2;\half12n+1,1-t^2)=t^{2-n}(t^2-1)^{n/2}
\end{equation}
and by \eqref{ident1a} $h_{2}(t)=f_{2}(t)/t$, therefore Theorem~\ref{th:N} yields
$$
\mathscr{N}_\alpha(u,v)=u\frac{h^2_2(t)} {f_2(t)}=u(1-\half{1}{t^2})^{n/2},
$$
where $t=t(v)$ is uniquely determined  by virtue of $f_{0}(t)=v$. Define
\begin{equation}\label{Ndef}
\M(v)=(1-t(v)^{-2})^{n/2}
\end{equation}
such that
$$
\mathscr{N}_\alpha(u,v)=u\M_n(v).
$$
To establish \eqref{equ:diff}, we notice that $\M_n(0)=0$ and also, using \eqref{coarea1}, one has from \eqref{Ndef}
$$
f'_{0}(t)=2th_2'(t)-f'_2(t)=f_2\left(\frac{2h_2'}{h_2}- \frac{f_2'}{f_2}\right) =nt^{1-n}(t^2-1)^{(n-2)/2},
$$
therefore,
$$
\frac{d\M_n}{dv}=\frac{d\M_n}{dt}\cdot \frac{1}{f_0'(t)}=\frac{1}{t^2}=1-\M_n^{2/n},
$$
as desired.

\end{proof}

\section{Proof of Theorem~\ref{th:main}}\label{sec:proofcor}
First we assume that dimension $n\ge2$ is chosen arbitrarily. Since $\alpha=1$, we have by the inversion invariance \eqref{ident1a} that
$f_{-1}(t)=f_1(t),$
which eliminates function $f_{-1}(t)$ from the consideration.
Then \eqref{coarea1} and \eqref{coarea2} amount to
\begin{equation}\label{f11}
\begin{split}
th_1'(t)&=f'_1(t),\\
(n-1)h_1(t)&=(t^2-1)f'_1-tf_1(t),
\end{split}
\end{equation}
where  by virtue of the hypergeometric representations \eqref{hyper1lin} and \eqref{hyperpsi2} we have respectively
\begin{align}
f_1(t)&=t^{2-n}(t^2-1)^{n/2}F(\half12, \half32;\half12n+1,1-t^2)\label{f1hyp}\\
h_1(t)&=t^{1-n}(t^2-1)^{n/2}F(\half{1}{2}, \half12;\half{n+2}2,1-t^2)\label{h1hyp}
\end{align}
Furthermore, by the definition $f_{1}(1)=h_{1}(1)=0$. Also, since $n\ge2$, one finds by l'Hospital's rule from  \eqref{diff1} that
\begin{equation}\label{diff2}
h_{1}'(1)=f_{1}'(1)=n\cdot\lim_{t\to 1+0}t^{2-n}(t^2-1)^{(n-2)/2}=\left\{
                               \begin{array}{ll}
                                 2, & \hbox{$n=2$;} \\
                                 0, & \hbox{$n\ge 3$.}
                               \end{array}
                             \right.
\end{equation}

Applying Theorem~\ref{th:N} we obtain

\begin{lemma}\label{lem:N1}
If $u,v>0$ then the supremum
$$
\mathscr{N}_1(u,v):=\sup\{\,|\mathcal{H}_1\rho|^2: \,\,\mathcal{F}_1\rho=u,
\,\,\,\,\mathcal{F}_{-1}\rho =v, \,\,\,\,0\le \rho\le 1\}
$$
is given by formula
$$
\mathscr{N}_1(u,v)=\Phi_n({uv})
$$
where $\Phi_n(s)$ is well-defined by
\begin{equation}\label{PhiDef}
\left\{\begin{array}{rl}
s&=t^{2-n}(t^2-1)^{n/2}F(\half12, \half32;\half12n+1,1-t^2)\\
\Phi_n(s)&=t^{-n-1}(t^2-1)^{n/2}F(\half12(n+1), \half32;\half12n+1,\frac{t^2-1}{t^2})
\end{array}
\right.
\end{equation}
Alternatively, $\Phi_n(s)$ is the unique solution of the singular initial problem \eqref{Phi1}--\eqref{Phi10}.
\end{lemma}

\begin{proof}
Since $f_{-1}=f_1$, it follows from Theorem~\ref{th:N} that $\mathscr{N}_1(u,v)=h_1(t)$, where $1< t<\infty$ is (uniquely) determined by $f_1(t)=\sqrt{uv}$, hence
$$
\mathscr{N}_1(u,v)=(h_1\circ f_1^{-1})(\sqrt{uv}).
$$
Let us consider the composed function $\Phi_n(s)=(h_1\circ f_1^{-1})(s)$. Obviously, $\Phi_n(s)$ is an increasing function. By the definition, $\Phi_n$ is determined by the parametric representation
\begin{equation}\label{paramPhi}
\left\{
\begin{array}{rl}
s&=f_1(t)\\
\Phi(s)&=h_1(t)
\end{array}
\right.
\end{equation}
therefore $\Phi_n(0)=0$  and by \eqref{hyper4},
$$
\lim_{z\to 1}\frac{f_1(t)}{t}=c_n:=\frac{\Gamma(\half{n+2}2)} {\Gamma(\half{n+1}2)\Gamma(\half32)}\equiv \frac{\omega_1\omega_{n-1}}{\omega_n}.
$$
Then it follows by l'Hospital's rule from \eqref{f11}${}_1$ that
$$
\lim_{s\to\infty}\frac{\Phi_n(s)}{\ln s}
=\lim_{t\to\infty}\frac{h_1(t)}{\ln f_1(t)}
=\lim_{t\to\infty}\frac{h'_1(t)f_1(t)}{f'_1(t)}
=\lim_{t\to\infty}\frac{f_1(t)}{t}=c_n,
$$
hence
$$
\Phi_n(s)\sim c_n \ln s\quad \text{ as }s\to \infty.
$$

Further, we find from  \eqref{f11}${}_1$ that $\Phi_n'(s)=h_1'(t)/f_1'(t)=1/t$, hence
$
\Phi_n'(0)=1
$
and applying the chain rule and  \eqref{f11}${}_2$ yields
$$
\Phi_n''(s)=-\frac{1}{t^2}\frac{dt}{ds}=-\frac{1}{t^2}\frac{1}{f_1'(t)}=
\frac{\frac{1}{t^2}-1}{(n-1)h_1(t)+tf_1(t))}
=\frac{\Phi_n'(\Phi_n'^2-1)}{(n-1)\Phi_n\Phi_n'+s}.
$$
This shows that $\Phi_n$ verifies the conditions \eqref{Phi1}--\eqref{Phi10}. It follows from the above pearametric representation that a solution of the singular problem \eqref{Phi1}--\eqref{Phi10} is unique. The lemma follows.
\end{proof}

We finish this section by some comments on the particular cases $n=1$ and $n=2$. If $n=1$ then analysis here is straightforward. Using \eqref{explicit1}, Theorem~\ref{th:N}  implies the following new inequality, see \eqref{tanh} and the discussion afterwards.

\begin{corollary}
\label{cor:N1a1}
For any measurable function $0\le \rho(x)\le 1$, $x\in \R{}$, and $0\not\in\supp\rho$, there holds
\begin{equation}\label{appl1}
\sinh^2\left(\frac12\int_{-\infty}^\infty \frac{\rho(x)}{x}dx\right)\le \frac14\int_{-\infty}^\infty \rho(x)dx \int_{-\infty}^\infty \frac{\rho(x)}{x^2}dx.
\end{equation}
The inequality is sharp and attained whenever $\rho(x)$ is a characteristic function of an interval $[a,b]$ with $ab>0$.
\end{corollary}

\begin{proof}
It follows from \eqref{explicit1} that $f_1(\xi)=\sinh \xi$ and $h_1(t)=\xi$, where $t=\cosh \xi$. In particular, taking into account that $d_\omega x=\frac12dx$, 
 \begin{equation}\label{appl1}
\frac12\int_{-\infty}^\infty \frac{\rho(x)}{x}dx\le \ln(\xi+\sqrt{1+\xi^2}),
\end{equation}
where
$$
\xi^2=\frac14\int_{-\infty}^\infty \rho(x)dx \int_{-\infty}^\infty \frac{\rho(x)}{x^2}dx.
$$
\end{proof}

Coming back to the  Markov type inequality \eqref{Lmoment1} involving $s_{-1}$, we obtain a further extension: in the notation of \eqref{sseq} we have
\begin{equation}\label{Lmoment2}
4\sinh^2 (\frac12s_{-1})\le s_0s_{-2}
\end{equation}

Next, note that the case $n=2$ is special in several respects. First, the derivative $f_1'(1)=2$ is nonzero by virtue of \eqref{diff2}. Furthermore, when $n=2$, the function $\Phi_2$ defined by \eqref{Phi1} has some extra symmetries due to $n-1=1$. In this case, the defining functions $f_1$ and $h_1$ admit a nice parameterizations by virtue of complete elliptic integrals. We confine ourselves  by the following remarkable Taylor expansion at the origin of $\Phi_2(z)$:
$$
\Phi_2(z)=
z-\half1{2^2}z^2+\half1{2^4}z^3-\half7{2^{9}}z^4+\half5{2^{11}}z^5- \half{21}{2^{16}}z^6+\half{3}{2^{17}}z^7+\half7{2^{24}}z^8 +\half{11}{2^{26}}z^9-\half{959}{2^{32}}z^{10}+\ldots
$$

\medskip

\bibliographystyle{abbrv}      

\end{document}